\documentclass{amsart}
\usepackage{amsmath}
\usepackage{amssymb}
\usepackage{amsthm}

\newtheorem{theorem}{Theorem}
\newtheorem{prop}{Proposition}
\newtheorem{lemma}{Lemma}
\newtheorem{rem}{Remark}
\newtheorem{exmp}{Example}

\begin{document}
\title[Commutativity preserving transformations of conjugacy classes]
{Commutativity preserving transformations on conjugacy classes of finite rank self-adjoint operators}
\author{Mark Pankov}
\subjclass[2000]{}

\address{Faculty of Mathematics and Computer Science, 
University of Warmia and Mazury, S{\l}oneczna 54, Olsztyn, Poland}
\email{pankov@matman.uwm.edu.pl}

\maketitle
\begin{abstract}
Let $H$ be a complex Hilbert space and let ${\mathcal C}$ be 
a conjugacy class of finite rank self-adjoint operators on $H$ with respect to the action of unitary operators.
We suppose that ${\mathcal C}$ is formed by operators of rank $k$ and for every $A\in {\mathcal C}$
the dimensions of distinct maximal eigenspaces are distinct.
Under the assumption that $\dim H\ge 4k$ we establish  that every bijective transformation $f$ of ${\mathcal C}$
preserving the commutativity in both directions is induced by a unitary or anti-unitary operator,
i.e. there is a unitary or anti-unitary operator $U$ such that $f(A)=UAU^{*}$ for every $A\in {\mathcal C}$.
A simple example shows that the condition concerning the dimensions of maximal eigenspaces cannot be omitted.
\end{abstract}

\section{Introduction}
The Grassmannian of $k$-dimensional subspaces of a complex Hilbert space $H$
is identified with the set ${\mathcal P}_k(H)$ of rank $k$ projections
(self-adjoint idempotents in the algebra of all bounded operators).
Projections play an important role in spectral theory of self-adjoint operators.
Projections of rank one correspond to pure states of quantum mechanical systems (Gleason's theorem).
Transformations of Grassmannians preserving various relations and structures are investigated in  
\cite{BJM, GS,Geher,Gyory,Molnar,Pankov1,Pankov2,Semrl,Stormer}.
In almost all cases, these transformations are induced by unitary and anti-unitary operators 
(or linear and conjugate-linear isometries for the non-surjective case). 
A large portion of these results  generalize  classical Wigner's theorem on symmetries of pure states (rank one projections).

Observe that each ${\mathcal P}_{k}(H)$ is a conjugacy class with respect to the action of unitary operators. 
It is natural to ask which of the above mentioned results can be extended on 
other conjugacy classes formed by finite rank self-adjoint operators?
Every such conjugacy class is completely determined by eigenvalues and 
the dimensions of maximal eigenspaces. 

Bijective transformations of ${\mathcal P}_k(H)$  preserving the orthogonality and commutativity (in both directions)
are described in \cite{Gyory,Semrl} and \cite{Pankov1}, respectively.
We consider such kind of transformations on other conjugacy classes of finite rank self-adjoint operators.
Our main result (Theorem 1) concerns commutativity preserving transformations.
A simple example shows that Gy\"ory--{\v S}emrl's theorem \cite{Gyory,Semrl} fails, 
but there is a weak version of this statement (Proposition 1) which will be used to prove the main result.
As in \cite{Pankov1}, we exploit some properties of orthogonal apartments 
(maximal sets of mutually commuting elements). 
In the case when the conjugacy class is different from ${\mathcal P}_k(H)$,  
orthogonal apartments have more complicated  intersections, 
but we will be able to characterize the orthogonality relation in terms of maximal intersections of orthogonal apartments. 
Also, in \cite{Pankov1} we essentially use Chow's theorem \cite{Chow}, 
but there is no analogue of this result for conjugacy classes distinct from ${\mathcal P}_k(H)$.

\section{Main result}
Let $H$ be a complex Hilbert space. 
Denote by ${\mathcal P}_{k}(H)$ the set of all rank $k$  projections 
and write ${\mathcal F}_{s}(H)$ for the real vector space formed by all self-adjoint operators of finite rank.
For every closed subspace $X\subset H$ the projection on $X$ will be denoted by $P_{X}$.

Let $A$ be a self-adjoint operator of finite rank whose non-zero eigenvalues are $\alpha_1,\dots, \alpha_{m}$
(each $\alpha_i$ is a real number). 
Then there is an orthonormal basis of $H$ consisting of eigenvectors of $A$. 
If $X_{i}$ is the maximal eigenspace  corresponding to the eigenvalue $\alpha_i$ 
(the subspace of all vectors $x\in H$ satisfying $A(x)=\alpha_i x$), then
$$A=\sum^{m}_{i=1}\alpha_{i}P_{X_i}$$
(we do not consider the kernel of $A$ as an eigenspace).
If $B$ is a finite rank self-adjoint  operator whose non-zero eigenvalues are $\beta_1,\dots, \beta_{p}$
and  $Y_{1},\dots,Y_p$ are the corresponding maximal eigenspaces,
then $A$ and $B$ are (unitary) conjugate, i.e. there is a unitary operator $U$ on $H$ such that
$$B=UAU^{*},$$
if and only if $m=p$ and 
$$\beta_{i}=\alpha_{\sigma(i)},\;\;\dim Y_{i}=\dim X_{\sigma(i)}$$
for a certain permutation $\sigma$ on the set $\{1,\dots,m\}$.
So, every conjugacy class ${\mathcal C}$ in ${\mathcal F}_{s}(H)$ is completely determined by the pair $(\alpha, d)$,
where $\alpha=\{\alpha_1,\dots, \alpha_m\}$ are non-zero eigenvalues of operators from ${\mathcal C}$
and  $d=\{d_1,\dots, d_{m}\}$ are the dimensions of maximal eigenspaces
such that
for every operator from ${\mathcal C}$ the maximal eigenspace corresponding to $\alpha_i$ is $d_i$-dimensional.

Two operators $A,B\in {\mathcal F}_{s}(H)$ are {\it orthogonal} if
$$AB=BA=0,$$
or equivalently, the images of $A$ and $B$ are orthogonal.
This is possible only in the case when the sum of the operator ranks is not greater than the dimension of $H$.
If $\dim H>2k$, then every bijective transformation of ${\mathcal P}_{k}(H)$
preserving the orthogonality relation in both directions is induced by a unitary or anti-unitary operator on $H$,
but this fails for other conjugacy classes of rank $k$ self-adjoint  ope\-ra\-tors 
(Example \ref{exmp-orth}).

Our main result is the following.

\begin{theorem}
Let ${\mathcal C}$ be a conjugacy class of rank $k$ operators from ${\mathcal F}_{s}(H)$
such that for every operator from ${\mathcal C}$ the dimensions of any two distinct maximal eigenspaces are distinct. 
Suppose that  $\dim H\ge 4k$
{\rm(}possibly $H$ is infinite-dimensional{\rm)}.
Let $f$  be a bijective transformation of ${\mathcal C}$ preserving the commutativity in both directions, i.e.
$$AB=BA\;\Longleftrightarrow\; f(A)f(B)=f(B)f(A)$$
for $A,B\in {\mathcal C}$. 
Then there is a unitary or anti-unitary operator $U$ on $H$ such that 
$$f(A)=UAU^{*}$$
for all $A\in {\mathcal C}$.
\end{theorem}

Note that the condition concerning the dimensions of maximal eigenspaces cannot be omitted
(Example \ref{exmp-com}).

\section{Orthogonality preserving transformations}
Let ${\mathcal G}_{k}(H)$ be  the Grassmannian consisting of all $k$-dimensional subspaces of $H$.
If $\dim H>2k$, then every bijective transformation of ${\mathcal G}_k(H)$ preserving the orthogonality relation in both 
directions is induced by a unitary or anti-unitary operator on $H$ \cite{GS,Gyory,Semrl}.
Using this fact, we prove the following.

\begin{prop}\label{prop-orth}
Let ${\mathcal C}$ be a conjugacy class of rank $k$ operators from ${\mathcal F}_{s}(H)$.
Suppose that $\dim H>2k$ and $f$ is a bijective transformation of ${\mathcal C}$ 
preserving the orthogonality relation in both directions, i.e.
$$AB=BA=0\;\Longleftrightarrow\; f(A)f(B)=f(B)f(A)=0$$
for $A,B\in {\mathcal C}$.
Then there is a unitary or anti-unitary operator $U$ such that 
$${\rm Im}(f(A))= U({\rm Im}(A))$$
for every $A\in {\mathcal C}$.
\end{prop}

\begin{proof}
For every closed subspace $X\subset H$ 
we denote by $[X]$ the set of all operators from ${\mathcal C}$ whose images are contained in $X$
(this set is empty if $\dim X <k$). 
Then $[X^{\perp}]$ consists of all operators from ${\mathcal C}$ orthogonal to every operator from $[X]$.
If $A\in {\mathcal C}$ is orthogonal to one operator from $[X]$, then it is orthogonal to all operators from $[X]$,
i.e. $A$ belongs to $[X^{\perp}]$. 

Let ${\mathcal L}$ be the partially ordered set formed by all closed subspaces of $H$
whose dimension and codimension both are not less than $k$
(by our assumption, $k$-dimensional subspaces belong to ${\mathcal L}$).
If $X\in {\mathcal L}$, then $X^{\perp}\in {\mathcal L}$ and each of the sets $[X]$ and $[X^{\perp}]$ is non-empty. 
Show that for every $X\in {\mathcal L}$ there is $g(X)\in {\mathcal L}$ such that $f([X])=[g(X)]$
and we have $g(X^{\perp})=g(X)^{\perp}$.

Consider the minimal subspace $Y$ with respect to the property that $f([X])\subset [Y]$.
Then $f([X^{\perp}])\subset [Y^{\perp}]$.
Since every operator from $[Y]$ is orthogonal to each operator from $[Y^{\perp}]$
and $f$ is orthogonality preserving in both directions, 
$f([X])=[Y]$ and $f([X^{\perp}])=[Y^{\perp}]$ which gives the claim.

So, $f$ induces a bijective transformation $g$ of ${\mathcal L}$.
For $X,Y\in {\mathcal L}$ 
$$X\subset Y\Leftrightarrow [X]\subset [Y]\Leftrightarrow f([X])\subset f([Y])\Leftrightarrow[g(X)]\subset [g(Y)]\Leftrightarrow
g(X)\subset g(Y)$$
which means that $g$ is an automorphism of the partially ordered set ${\mathcal L}$.
Then $g({\mathcal G}_{k}(H))={\mathcal G}_{k}(H)$. 
Since $g$ is orthogonality preserving in both directions,
there is a unitary or anti-unitary operator $U$ such that $g(X)=U(X)$ for every $X\in  {\mathcal G}_{k}(H)$.
On the other hand, the image of every $A\in {\mathcal C}$ is $k$-dimensional and 
$${\rm Im}(f(A))=g({\rm Im}(A))$$
which implies the required equality. 
\end{proof}

\begin{exmp}\label{exmp-orth}{\rm
As above, we suppose that ${\mathcal C}$ is a conjugacy class of rank $k$ operators from ${\mathcal F}_{s}(H)$
and $\dim H>2k$. 
Assume also that ${\mathcal C}\ne {\mathcal P}_{k}(H)$ and 
take any $k$-dimen\-sional subspace $X\subset H$.
There are distinct operators $A,B\in {\mathcal C}$ whose images coincide with $X$. 
Let $f$ be the transformation of ${\mathcal C}$ which transposes $A,B$ and leaves fixed all other operators from ${\mathcal C}$.
It preserves the orthogonality relation  in both directions 
($C\in {\mathcal C}$ is orthogonal to $A$ if and only if it is orthogonal to $B$),
but there is no unitary or anti-unitary operator $U$
satisfying $f(C)=UCU^{*}$ for all $C\in {\mathcal C}$.
}\end{exmp}

\section{Geometric interpretation of commutativity} 
Two projections $P_{X}$ and $P_{Y}$ commute if and only if the subspaces $X$ and $Y$ are compatible.
The latter means that the subspace 
$$(X\cap Y)^{\perp}\cap X\;\mbox{ and }\;(X\cap Y)^{\perp}\cap Y$$
are orthogonal.
For example, two closed subspaces are compatible if they are orthogonal or 
one of these subspaces is contained in the other.
In general, two closed subspaces of $H$ are compatible if and only if there is an orthonormal basis of $H$
such that each of these subspaces is spanned by a subset of this basis. 

Let $A$ and $B$ be finite rank self-adjoint operators whose maximal eigenspaces are
$X_{1},\dots,X_{m}$ and $Y_1,\dots, Y_p$ (respectively).
Then $A$ and $B$ commute if and only if each pair $P_{X_i},P_{Y_j}$ commute,
i.e. every $X_{i}$ is compatible to all $Y_j$.
This is a simple case of von Neumann theorem on projection-valued measures \cite[Section VI.2]{Var}. 

We say that a set formed by closed subspaces of $H$ is {\it compatible} if any two elements of this set are compatible.
For any orthonormal basis $B$ of $H$ the set of all closed subspaces spanned by subsets of $B$
is called the {\it orthogonal apartment} associated to the basis $B$. 
Every  orthogonal apartment is a compatible set.

\begin{prop}\label{prop-ap}
Every compatible set consisting of finite-dimensional subspaces is contained in an orthogonal apartment.
\end{prop}

\begin{proof}
We will use the following properties of the compatibility relation:
\begin{enumerate}
\item[{\rm (1)}] 
If closed subspaces $X,Y\subset H$ are compatible, then $X$ and $Y^{\perp}$ are compatible. 
\item[{\rm (2)}] 
If ${\mathcal Z}$ is a set formed by closed subspaces of $H$
and $X$ is a closed subspace of $H$ compatible to all elements of ${\mathcal Z}$,
then $X$ is compatible to $\bigcap_{Z\in {\mathcal Z}}Z$ and the smallest closed subspace containing all 
elements of ${\mathcal Z}$.
\end{enumerate}
See \cite[Lemma 3.10]{Var}.

Let ${\mathcal X}$ be a compatible set formed by finite-dimensional subspaces.
First, we prove the statement for the case when $H$ is finite-dimensional.
Consider a set  ${\mathcal Y}$ which consists of mutually orthogonal subspaces  whose sum coincides with $H$ and 
such that every element of ${\mathcal Y}$ is compatible to all elements of ${\mathcal X}$.
For example, for every $X\in {\mathcal X}$ the set $\{X,X^{\perp}\}$ satisfies this condition.
If there is $X\in {\mathcal X}$ which intersects a certain $Y\in {\mathcal Y}$ in a proper subspace of $Y$,
then 
$$Y'=Y\cap X\;\mbox{ and }\; Y''=(Y\cap X)^{\perp}\cap Y$$
are orthogonal subspaces whose sum coincides with $Y$ and each of these subspaces 
is compatible to all elements of ${\mathcal X}$.
In this case, we replace ${\mathcal Y}$ by the set
$$({\mathcal Y}\setminus \{Y\})\cup \{Y',Y''\}.$$
Since $H$ is finite-dimensional, 
we can construct  recursively a set ${\mathcal Y}$ satisfying the above conditions and such that 
for all $X\in {\mathcal X}$ and $Y\in {\mathcal Y}$ we have either  $X\cap Y=Z$ or $X\cap Y=0$.
Then every $X\in {\mathcal X}$ is the sum of all elements of ${\mathcal Y}$ contained in $X$.
There is an orthonormal basis of $H$ such that every element of ${\mathcal Y}$ is spanned by a subset of this basis.
The associated orthogonal apartment contains ${\mathcal X}$. 

Now, let $H$ be infinite-dimensional.
Consider the family of all sets  formed by mutually orthogonal finite-dimensional subspaces of $H$
compatible to all elements of ${\mathcal X}$.
Using Zorn Lemma, we establish the existence of a maximal set ${\mathcal Y}$ with respect to this property.
Let $H'$ be the smallest closed subspace containing all elements of ${\mathcal Y}$.
This subspace is compatible to all elements of ${\mathcal X}$.
Show that $H'=H$.

Suppose that $H'$ is a proper subspace of $H$.
If all elements of ${\mathcal X}$ are contained in $H'$, 
then we add to ${\mathcal Y}$ any non-zero finite-dimensional subspace of $H'^{\perp}$
and get a set whose elements are mutually orthogonal and compatible to all elements of ${\mathcal X}$.
This is impossible, since ${\mathcal Y}$ is maximal with respect to this property. 
If there is $X\in {\mathcal X}$ which is not contained in $H'$, 
then the non-zero subspace $(X\cap H')^{\perp}\cap X$ 
is compatible to all elements of ${\mathcal X}$ and orthogonal to all elements of ${\mathcal Y}$.
As above, we add this subspace to ${\mathcal Y}$ and obtain a set whose elements are mutually orthogonal and compatible to 
all elements of ${\mathcal X}$, a contradiction.

For every $Y\in {\mathcal Y}$ we denote by ${\mathcal X}_{Y}$
the set formed by all intersections of $Y$ with elements of ${\mathcal X}$. 
This set is compatible (possibly ${\mathcal X}_{Y}\subset \{0,Y\}$).
Since every $Y\in {\mathcal Y}$ is finite-dimensional, 
there is an orthonormal basis $B_{Y}$ of $Y$ such that all elements of ${\mathcal X}_{Y}$ are spanned by subsets of this basis
(we can take any orthonormal basis of $Y$ if ${\mathcal X}_{Y}\subset \{0,Y\}$).
Let $B$ be the union of all $B_Y$. This is an orthonormal basis of $H$.
For every $X\in {\mathcal X}$ the intersection $X\cap Y$ is non-zero only for finitely many $Y\in {\mathcal Y}$ 
and the sum of all such $X\cap Y$ coincides with $X$
(this follows from the fact that $X$ is compatible to all elements of ${\mathcal Y}$).
This means that ${\mathcal X}$ is contained in the orthogonal apartment defined by the basis $B$.
\end{proof}

Let ${\mathcal C}$ be a conjugacy class of operators from ${\mathcal F}_{s}(H)$.
For every orthonormal basis $B$ of $H$ the set of all operators from ${\mathcal C}$
whose maximal eigenspaces are spanned by subsets of $B$ is said to be 
the {\it orthogonal apartment} of ${\mathcal C}$ defined by the basis $B$.
Any two operators from an orthogonal apartment commute.
Conversely,
if ${\mathcal X}$ is a subset of ${\mathcal C}$ formed by mutually commuting operators,
then, by  Proposition \ref{prop-ap}, there is an orthogonal apartment of ${\mathcal C}$ containing ${\mathcal X}$.
Therefore, orthogonal apartments of ${\mathcal C}$ can be characterized as 
maximal subsets ${\mathcal X}\subset {\mathcal C}$ 
with respect to the property that any two operators from ${\mathcal X}$ commute.
Therefore, for a bijective transformation $f$ of ${\mathcal C}$ the following two conditions are equivalent:
\begin{enumerate}
\item[$\bullet$] $f$ is commutativity preserving in both directions,
\item[$\bullet$] $f$ and $f^{-1}$ send orthogonal apartments to orthogonal apartments. 
\end{enumerate}
Using some properties of orthogonal apartments, we prove the following.

\begin{lemma}\label{lemma-main}
If $\dim H\ge 4k$,
then every bijective transformation of ${\mathcal C}$ preserving the commutativity in both directions is
orthogonality preserving in both directions.
\end{lemma}

Now, we show that Theorem 1 follows from Proposition \ref{prop-orth} and Lemma \ref{lemma-main}.

Let $f$ be a bijective  transformation of ${\mathcal C}$ preserving the commutativity in both directions
and  $\dim H\ge 4k$.
By Lemma \ref{lemma-main}, $f$ is orthogonality preserving in both directions
and Pro\-po\-sition \ref{prop-orth} implies the existence of a unitary or anti-unitary operator $U$ such that
$${\rm Im}(f(A))= U({\rm Im}(A))$$
for every $A\in {\mathcal C}$.
Consider the transformation $g$ of ${\mathcal C}$ defined as
$$g(A)= U^{*}f(A)U$$
for every $A\in {\mathcal C}$. This is a bijective transformation which preserves the commutativity in both directions
and leaves fixed the image of every operator from ${\mathcal C}$.
So, the image of every $A\in {\mathcal C}$ coincides with the image of $g(A)$ 
which implies that $g(A)=A$ only in the case when ${\mathcal C}={\mathcal P}_{k}(H)$.

Show that every $1$-dimensional eigenspace $Y$ of $A$ is an eigenspace of $g(A)$
(recall that we do not consider eigenspaces contained in the kernels of operators).
We take a $k$-dimensional subspace $N$ compatible to ${\rm Im}(A)$ and intersecting ${\rm Im}(A)$ precisely in $Y$.
There is an operator $B\in {\mathcal C}$ whose image is $N$ and such that $AB=BA$.
Then $g(A)$ and $g(B)$ commute and the image of $g(B)$ is $N$.
If $Y$ is not an eigenspace of $g(A)$, then the intersection of every maximal eigenspace $X'$ of $g(A)$ with
every maximal eigenspace $Y'$ of $g(B)$ is zero.
Since $g(A)$ and $g(B)$ commute, $X'$ and $Y'$ are compatible which implies that $X'$ and $Y'$ are orthogonal.
Therefore, all maximal eigenspaces of $g(A)$ are orthogonal to all maximal eigenspaces of $g(B)$
which means that the images of $g(A)$ and $g(B)$ are orthogonal and we get a contradiction
(the intersection of these images is $Y$).

Similarly, we establish that every $1$-dimensional eigenspace of $g(A)$ is an eigenspa\-ce of $A$.
So, a $1$-dimensional subspace is an eigenspace of $A$ if and only if it is an eigenspace of $g(A)$.
This is possible only in the case when every maximal eigenspace of $A$ is a maximal eigenspace of $g(A)$
and conversely. 
By our assumption (see Theorem 1), the dimensions of any two distinct maximal eigenspaces of $A$ are distinct 
and the same holds for the maximal eigenspaces of $g(A)$. 
In other words, 
every maximal eigenspace of $A$ and $g(A)$ corresponds to the same eigenvalue as an eigenspace of $A$ 
and as an eigenspace of $g(A)$. Then $g(A)=A$
and we have $f(A)=UAU^{*}$ for every $A\in {\mathcal C}$.

\begin{exmp}\label{exmp-com}{\rm
Suppose that every operator from ${\mathcal C}$ has precisely two eigenvalues $\alpha,\beta$
and the maximal eigenspace corresponding to each of these eigenvalues is $m$-dimensional.
If $X$ and $Y$ are orthogonal $m$-dimensional subspaces, then the operators
$$A=\alpha P_{X}+\beta P_{Y}\;\mbox{ and }\; B=\alpha P_{Y}+\beta P_{X}$$
belong to ${\mathcal C}$. 
Let $f$ be the bijective transformation of ${\mathcal C}$ which permutes $A,B$ and leaves fixed all other elements of 
${\mathcal C}$. 
Then $f$ is commutativity preserving in both directions. 
Indeed, $C\in {\mathcal C}$ commutes with $A$ if and only if it commutes with $B$.
On the other hand, there is no unitary or anti-unitary operator $U$ such that
$f(C)=UCU^{*}$ for every $C\in {\mathcal C}$.
}\end{exmp}

\section{Proof of Lemma \ref{lemma-main}}
Let ${\mathcal C}$ be a conjugacy class of rank $k$ operators from ${\mathcal F}_{s}(H)$
and $\dim H>2k$.
We take any  orthonormal basis $B=\{e_{i}\}_{i\in I}$ of $H$
and consider the associated orthogonal apartment ${\mathcal A}\subset {\mathcal C}$. 
Recall that ${\mathcal A}$ consists of all operators from ${\mathcal C}$
whose maximal eigenspaces are spanned  by subsets of $B$.

A subset ${\mathcal X}\subset{\mathcal A}$ is said to be {\it orthogonally inexact}
if there is an orthogonal apartment of ${\mathcal C}$ distinct from ${\mathcal A}$ and containing ${\mathcal X}$.

\begin{exmp}{\rm
For any distinct $i,j\in I$ we denote by ${\mathcal A}(+i,+j)$ 
the set of all operators $A\in {\mathcal A}$ such that one of the maximal eigenspaces of $A$ contains both $e_i$ and $e_j$. 
Let ${\mathcal A}(-i,-j)$ be the set of all operators from ${\mathcal A}$ whose images 
are orthogonal to the $2$-dimensional subspace $S$ spanned by $e_i$ and $e_j$.
We take any orthogonal unit vectors $e'_{i},e'_{j}\in S$ which are not scalar multiples of $e_i,e_j$.
Then the orthogonal apartment defined by the basis $$(B\setminus \{e_i,e_j\})\cup\{e'_i,e'_j\}$$
intersects ${\mathcal A}$ precisely in the subset
\begin{equation}\label{eq1}
{\mathcal A}(+i,+j)\cup {\mathcal A}(-i,-j),
\end{equation}
i.e. this subset is orthogonally inexact. 
}\end{exmp}

\begin{lemma}\label{lemma2}
Every maximal orthogonally inexact subset of ${\mathcal A}$ is of type \eqref{eq1}.
\end{lemma}

\begin{proof}
It is sufficient to show that every orthogonally inexact subset ${\mathcal X}\subset {\mathcal A}$
is contained in a subset of type \eqref{eq1}.
For every $i\in I$ we consider the family of all closed subspaces $X\subset H$ satisfying one of the following conditions:
\begin{enumerate}
\item[$\bullet$]  
$X$ is the maximal eigenspace of an operator from ${\mathcal X}$ and $X$ contains $e_{i}$,
\item[$\bullet$]  
there is $A\in {\mathcal X}$ whose image is orthogonal to $e_{i}$ and $X$ is the orthogonal complement of this image.
\end{enumerate}
Note that every such $X$ is spanned by a subset of $B$ and contains $e_i$.
Let $S_i$ be the intersection of all subspaces $X$ satisfying the above conditions. 
Then $S_i$ is spanned by a subset of $B$ and contains $e_i$.
Therefore, if every $S_{i}$ is $1$-dimensional, 
then ${\mathcal A}$ is the unique orthogonal apartment containing ${\mathcal X}$
which contradicts the assumption that ${\mathcal X}$ is orthogonally inexact.
So, there is at least one $i\in I$ such that $\dim S_{i}\ge 2$
and we take any $j\ne i$ satisfying $e_{j}\in S_{i}$.
Then
$${\mathcal X}\subset {\mathcal A}(+i,+j)\cup {\mathcal A}(-i,-j).$$
Indeed, if $A\in {\mathcal X}$ and $X$ is a maximal eigenspace of $A$ containing $e_{i}$, 
then $e_{j}\in S_{i}\subset X$
and $A$ belongs to ${\mathcal A}(+i,+j)$. 
If the image of $A\in {\mathcal X}$ is orthogonal to $e_{i}$, then $e_{j}\in S_{i}\subset {\rm Im}(A)^{\perp}$
which means that $e_{j}$ is orthogonal to ${\rm Im}(A)$
and $A$ belongs to ${\mathcal A}(-i,-j)$. 
\end{proof}

We say that ${\mathcal C}\subset {\mathcal A}$ is an {\it orthocomplementary} subset
if ${\mathcal A}\setminus {\mathcal C}$ is a maximal orthogonally inexact subset, i.e.
\begin{equation}\label{eq2}
{\mathcal C}={\mathcal A}\setminus ({\mathcal A}(+i,+j)\cup {\mathcal A}(-i,-j))
\end{equation}
for some distinct $i,j\in I$.

Denote by ${\mathcal A}(+i,-j)$ the set of all operators $A\in {\mathcal A}$
such that one of the maximal eigenspaces of $A$ contains $e_i$ and does not contain $e_{j}$.

\begin{rem}{\rm
If $A\in {\mathcal A}(+i,-j)$, then one of the following possibilities is realized:
$e_i,e_j$ belong to distinct maximal eigenspaces of $A$ or $e_{j}$ is orthogonal to the image of $A$.
Therefore, 
$${\mathcal A}(+i,-j)\cap {\mathcal A}(+j,-i)$$
consists of all operators $A\in {\mathcal A}$ such that $e_{i}$ and $e_{j}$ belong to 
distinct maximal eigenspaces of $A$. 
In the case when ${\mathcal C}={\mathcal P}_k(H)$, this intersection is empty. 
}\end{rem}

The orthocomplementary subset \eqref{eq2} coincides with 
$${\mathcal A}(+i,-j)\cup {\mathcal A}(+j,-i).$$
This orthocomplementary subset will be denoted by ${\mathcal C}_{ij}$.
Note that ${\mathcal C}_{ij}={\mathcal C}_{ji}$.

For any $A,B\in {\mathcal A}$ we denote by $n_{\mathcal A}(A,B)$ 
the number of all orthocomplementary subsets of ${\mathcal A}$ containing both $A$ and $B$.

\begin{lemma}\label{lemma-3}
Let $A,B\in {\mathcal A}$. The following assertions are fulfilled:
\begin{enumerate}
\item[(i)]  Suppose that $H$ is infinite-dimensional.
If $A,B$ are orthogonal, then $$n_{\mathcal A}(A,B)=k^2.$$
In the case when $A,B$ are non-orthogonal, $n_{\mathcal A}(A,B)$ is infinite.
 \item[(ii)] If $\dim H=n$ is finite and the intersection of the images of $A$ and $B$
is $m$-dimensional, then 
$$n_{\mathcal A}(A,B)\ge (k-m)^{2}+m(n-2k+m).$$
\end{enumerate}
\end{lemma}

\begin{proof}
Let $X$ and $Y$ be the images of $A$ and $B$, respectively.
In the case when $X$ and $Y$ are orthogonal,
${\mathcal C}_{ij}$ contains both $A,B$ if and only if one of $e_{i},e_{j}$ belongs to $X$ and the other to $Y$.
There are precisely $k^2$ such orthocomplementary subsets.

Suppose that $X\cap Y$ is $m$-dimensional and $m>0$.
If $$e_{i}\in X\cap Y\;\mbox{ and }\;e_{j}\not\in X+Y,$$ 
then ${\mathcal C}_{ij}$ contains both $A,B$. 
In the case when $H$ is infinite-dimensional, there are infinitely many such ${\mathcal C}_{ij}$.
If $\dim H=n$ is finite, then 
$$\dim (X+Y)=2k-m$$
and there are precisely $m(n-2k+m)$ such orthocomplementary subsets.

If one of $e_{i},e_{j}$ belongs to $X\setminus Y$ and  the other to $Y\setminus X$, 
then ${\mathcal C}_{ij}$ contains both $A,B$.
The number of such ${\mathcal C}_{ij}$ is $(k-m)^2$ .
\end{proof}

\begin{rem}\label{rem2}{\rm
Suppose that there are $e_{i}\in X\cap Y$ and $e_{j}\in X\setminus Y$ which are not contained in the same eigenspace of $A$.
Then ${\mathcal C}_{ij}$ contains both $A,B$ and we have
$$n_{\mathcal A}(A,B)> (k-m)^{2}+m(n-2k+m).$$
}\end{rem}

\begin{lemma}\label{lemma-4}
In the case when $\dim H\ge 4k$,
operators $A,B\in {\mathcal A}$ are orthogonal if and only if $n_{\mathcal A}(A,B)=k^2$.
\end{lemma}

\begin{proof}
If $H$ is infinite-dimensional, then the statement follows immediately from the first part of Lemma \ref{lemma-3}.
Suppose that $\dim H=n$ is finite and consider the quadratic function 
$$c(x)=(k-x)^{2}+x(n-2k+x)=2x^{2}-(4k-n)x+k^2.$$
It takes the minimal value on $x=(4k-n)/4$.
Since $n\ge 4k$, we have $(4k-n)/4\le 0$.
Then for every natural $m>0$
$$k^2=c(0)<c(m)=(k-m)^{2}+m(n-2k+m)$$
and Lemma \ref{lemma-3} gives the claim.
\end{proof}

\begin{rem}{\rm
Suppose that $\dim H=n$ is even and $2k <n<4k$. 
Them $m=(4k-n)/2$ is a natural number and $0<m<k$.
A direct verification shows that $$c(0)=c(m)$$
and we cannot state that $n_{\mathcal A}(A,B)\ne k^2$ 
if the intersection of the images of $A$ and $B$ is $m$-dimensional.
There is another one problem. 
In the case when $\dim H<4k$, we have 
$$c(0)>c(m)$$ for some $m\in \{1,\dots,k-1\}$.
By Remark \ref{rem2}, there are $A,B\in {\mathcal A}$ 
(the intersection of the images is assumed to be $m$-dimensional)
such that $$n_{\mathcal A}(A,B)>c(m).$$
As above, we cannot state that $n_{\mathcal A}(A,B)\ne k^2$.
}\end{rem}

Let $f$ be a bijective transformation of ${\mathcal C}$ preserving the commutativity in both directions.
Then $f$ and $f^{-1}$ send orthogonal apartments to orthogonal apartments.
For any orthogonal operators $A,B\in {\mathcal C}$ there is an orthogonal apartment ${\mathcal A}\subset {\mathcal C}$
containing them.
The transformation $f$ sends orthogonally inexact subsets of ${\mathcal A}$ 
to orthogonally inexact subsets of $f({\mathcal A})$
and $f^{-1}$ maps orthogonally inexact subsets of $f({\mathcal A})$ to orthogonally inexact subsets of ${\mathcal A}$.
This means that ${\mathcal X}$ is a maximal orthogonally inexact subset of ${\mathcal A}$
if and only if $f({\mathcal X})$ is a maximal orthogonally inexact subset of $f({\mathcal A})$.
Therefore, a subset ${\mathcal C}\subset {\mathcal A}$ is orthocomplementary if and only if 
$f({\mathcal C})$ is an orthocomplementary subset of $f({\mathcal A})$.
Since $A$ and $B$ are orthogonal, we have
$$n_{f({\mathcal A})}(f(A),f(B))=n_{\mathcal A}(A,B)=k^2.$$
In the case when $\dim H\ge 4k$, the operators $f(A)$ and $f(B)$ are orthogonal by Lemma \ref{lemma-4}.
Similarly, we establish that $f^{-1}$ is orthogonality preserving.

\end{document}